\documentclass[11pt,a4]{amsart}
\usepackage[centertags]{amsmath}
\usepackage{amsfonts}
\usepackage{amssymb}
\usepackage{amsthm}
\usepackage{newlfont}
\newcommand{\Z}{\Bbb Z}
\newcommand{\N}{\Bbb N}

\def\NL{\hfill\break}

\newcommand{\set}[1]{\left\{#1\right\}}

\newtheorem{theorem}{Theorem}[section]
\newtheorem{definition}[theorem]{Definition}
\newtheorem{proposition}[theorem]{Proposition}
\newtheorem{corollary}[theorem]{Corollary}
\newtheorem{lemma}[theorem]{Lemma}
\newtheorem{remark}[theorem]{Remark}
\newtheorem{example}[theorem]{Example}

\begin{document}
\title[On $(\sigma,\delta)$-skew McCoy modules]{On $(\sigma,\delta)$-skew McCoy modules}
\date{\today}            

\subjclass[2010]{16S36, 16U80}
\keywords{McCoy module, $(\sigma,\delta)$-skew McCoy module, semicommutative module, Armendariz module, $(\sigma,\delta)$-skew Armendariz module, reduced module}

\author[M. Louzari]{Mohamed Louzari}
\address{Depertment of Mathematics\\
  Faculty of sciences \\
  Abdelmalek Essaadi University\\
  BP. 2121 Tetouan, Morocco}
\email{mlouzari@yahoo.com}

\author[L. Ben Yakoub]{L'moufadal Ben Yakoub}
\address{Depertment of Mathematics\\
  Faculty of sciences \\
  Abdelmalek Essaadi University\\
  BP. 2121 Tetouan, Morocco}
\email{benyakoub@hotmail.com}

\begin{abstract}Let $(\sigma,\delta)$ be a quasi derivation of a ring $R$ and $M_R$ a right $R$-module. In this paper, we introduce the notion of $(\sigma,\delta)$-skew McCoy modules which extends the notion of McCoy modules and $\sigma$-skew McCoy modules. This concept can be regarded also as a generalization of $(\sigma,\delta)$-skew Armendariz modules. Some properties of this concept are established and some connections between $(\sigma,\delta)$-skew McCoyness and $(\sigma,\delta)$-compatible reduced modules are examined. Also, we study the property $(\sigma,\delta)$-skew McCoy of some skew triangular matrix extensions $V_n(M,\sigma)$, for any nonnegative integer $n\geq 2$. As a consequence, we obtain: (1) $M_R$ is $(\sigma,\delta)$-skew McCoy if and only if $M[x]/M[x](x^n)$ is $(\overline{\sigma},\overline{\delta})$-skew McCoy, and (2) $M_R$ is $\sigma$-skew McCoy if and only if $M[x;\sigma]/M[x;\sigma](x^n)$ is $\overline{\sigma}$-skew McCoy.

\end{abstract}

\maketitle

\section{Introduction}  Throughout this paper, $R$ denotes an associative ring with unity and $M_R$ a right $R$-module. For a subset $X$ of a module $M_R$, $r_R(X)=\{a\in R|Xa=0\}$ and $\ell_R(X)=\{a\in R|aX=0\}$ will stand for the right and the left annihilator of $X$ in $R$ respectively. An Ore extension of a ring $R$ is denoted by $R[x;\sigma,\delta]$, where $\sigma$ is an endomorphism of $R$ and $\delta$ is a $\sigma$-derivation, i.e., $\delta\colon R\rightarrow R$ is an additive map such that $\delta(ab)=\sigma(a)\delta(b)+\delta(a)b$ for all $a,b\in R$ (the pair $(\sigma,\delta)$ is also called a quasi-derivation of $R$). Recall that elements of $R[x;\sigma,\delta]$ are polynomials in $x$ with coefficients written on the left. Multiplication in $R[x;\sigma,\delta]$ is given by the multiplication in $R$ and the condition $xa=\sigma(a)x+\delta(a)$, for all $a\in R$. In the next, $S$ will stand for the Ore extension $R[x;\sigma,\delta]$. On the other hand, we have a natural functor $-\otimes_RS$ from the category of right $R$-modules into the category of right $S$-modules. For a right $R$-module $M$, the right $S$-module $M\otimes_R S$ is called {\it the induced module} \cite{matczuk/induced}. Since $R[x;\sigma,\delta]$ is a free left $R$-module, elements of $M\otimes_R S$ can be seen as polynomials in $x$ with coefficients in $M$ with natural addition and right $S$-module multiplication.

\par For any $0\leq i\leq j\;(i,j\in \N)$, $f_i^j\in End(R,+)$ will denote the map which is the sum of all possible words in $\sigma,\delta$ built with $i$ factors of $\sigma$ and $j-i$ factors of $\delta$ (e.g., $f_n^n=\sigma^n$ and $f_0^n=\delta^n, n\in \N $). We have $x^ja=\sum_{i=0}^jf_i^j(a)x^i$ for all $a\in R$, where $i,j$ are nonnegative integers with $j\geq i$ (see \cite[Lemma 4.1]{lam}).

\par Following Lee and Zhou \cite{lee/zhou}, we introduce the notation $M[x;\sigma,\delta]$ to write the $S$-module $M\otimes_R S$. Consider $$M[x;\sigma,\delta]:=\set{\sum_{i=0}^nm_ix^i\mid n\geq 0,m_i\in M};$$ which is an $S$-module under an obvious addition and the action of monomials of $R[x;\sigma,\delta]$ on monomials in $M[x;\sigma,\delta]_{R[x;\sigma,\delta]}$ via $(mx^j)(ax^{\ell})=m\sum_{i=0}^jf_i^j(a)x^{i+\ell}$ for all $a\in R$ and $j,\ell\in \N$. The $S$-module $M[x;\sigma,\delta]$ is called the {\it skew polynomial extension} related to the quasi-derivation $(\sigma,\delta)$.

\par A module $M_R$ is semicommutative, if for any $m\in M$ and $a\in R$, $ma=0$ implies $mRa=0$ \cite{rege2002}. Let $\sigma$ an endomorphism of $R$, $M_R$ is called an $\sigma$-semicommutative module \cite{zhang/chen} if, for any $m\in M$ and $a\in R$, $ma=0$ implies $mR\sigma(a)=0$. For a module $M_R$ and a quasi-derivation $(\sigma,\delta)$ of $R$, we say that $M_R$ is $\sigma$-compatible, if for each $m\in M$ and $a\in R$, we have $ma=0 \Leftrightarrow m\sigma(a)=0$. Moreover, we say that $M_R$ is $\delta$-compatible, if for each $m\in M$ and $a\in R$, we have $ma=0\Rightarrow m\delta(a)$=0. If $M_R$ is both $\sigma$-compatible and $\delta$-compatible, we say that $M_R$ is $(\sigma,\delta)$-compatible (see \cite{annin/2004}). In \cite{zhang/chen}, a module $M_R$ is called $\sigma$-{\it skew Armendariz}, if $m(x)f(x)=0$ where $m(x)=\sum_{i=0}^nm_ix^i\in M[x;\sigma]$ and $f(x)=\sum_{j=0}^ma_jx^j\in R[x;\sigma]$ implies $m_i\sigma^i(a_j)=0$ for all $i,j$. According to Lee and Zhou \cite{lee/zhou}, $M_R$ is called $\sigma$-{\it Armendariz}, if it is $\sigma$-compatible and $\sigma$-skew Armendariz.

\par Following Alhevas and Moussavi \cite{moussavi/2012}, a module $M_R$ is called $(\sigma,\delta)$-skew Armendariz, if whenever $m(x)g(x)=0$ where $m(x)=\sum_{i=0}^pm_ix^i\in M[x;\sigma,\delta]$ and $g(x)=\sum_{j=0}^qb_jx^j\in R[x;\sigma,\delta]$, we have $m_ix^ib_jx^j=0$ for all $i,j$.


In this paper, we introduce the concept of $(\sigma,\delta)$-skew McCoy modules which is a generalization of McCoy modules and $\sigma$-skew McCoy modules. This concept can be regarded also as a generalization of $(\sigma,\delta)$-skew Armendariz modules and rings. We study connections between reduced modules, $(\sigma,\delta)$-compatible modules and $(\sigma,\delta)$-skew McCoy modules. Also, we show that $(\sigma,\delta)$-skew McCoyness passes from a module $M_R$ to its skew triangular matrix extension $V_n(M,\sigma)$. In this sens, we complete the definition of skew triangular matrix rings $V_n(R,\sigma)$ given by Isfahani \cite{isfahani/2011}, by introducing the notion of skew triangular matrix modules. Moreover, we give some results on $(\sigma,\delta)$-skew McCoyness for skew triangular matrix modules.

\section{$(\sigma,\delta)$-skew McCoy modules}%

\par Cui and Chen \cite{cui/2011,cui/2012}, introduced both concepts of McCoy modules and $\sigma$-skew McCoy modules. A module $M_R$ is called {\it McCoy} if $m(x)g(x)=0$, where $m(x)=\sum_{i=0}^pm_ix^i\in M[x]$ and $g(x)=\sum_{j=0}^qb_jx^j\in R[x]\setminus\{0\}$ implies that there exists $a\in R\setminus\{0\}$ such that $m(x)a=0$. A module $M_R$ is called {\it $\sigma$-skew McCoy} if $m(x)g(x)=0$, where $m(x)=\sum_{i=0}^pm_ix^i\in M[x;\sigma]$ and $g(x)=\sum_{j=0}^qb_jx^j\in R[x;\sigma]\setminus\{0\}$ implies that there exists $a\in R\setminus\{0\}$ such that $m(x)a=0$. With the same manner, we introduce the concept of {\it $(\sigma,\delta)$-skew McCoy} modules which is a generalization of McCoy modules, $\sigma$-skew McCoy modules and $(\sigma,\delta)$-skew Armendariz modules.

\begin{definition}Let $M_R$ be a module and $M[x;\sigma,\delta]$ the corresponding $(\sigma,\delta)$-skew polynomial module over $R[x;\sigma,\delta]$.
\par$\mathbf{(1)}$ The module $M_R$ is called {\it $(\sigma,\delta)$-skew McCoy} if $m(x)g(x)=0$, where $m(x)=\sum_{i=0}^pm_ix^i\in M[x;\sigma,\delta]$ and $g(x)=\sum_{j=0}^qb_jx^j\in R[x;\sigma,\delta]\setminus\{0\}$, implies that there exists $a\in R\setminus\{0\}$ such that $m(x)a=0$ $($i.e., $\sum_{i=\ell}^pm_if_{\ell}^i(a)=0$, for all $\ell=0,1,\cdots,p)$.
\par$\mathbf{(2)}$ The ring $R$ is called {\it $(\sigma,\delta)$-skew McCoy} if $R$ is $(\sigma,\delta)$-skew McCoy as a right $R$-module.
\end{definition}

\begin{remark}\label{rem2}$\mathbf{(1)}$ If $M_R$ is an $(\sigma,\delta)$-skew Armendariz module then it is $(\sigma,\delta)$-skew McCoy $($Proposition \ref{prop1}$)$. But the converse is not true $($Example \ref{exp mcnotarm}$)$.
\par$\mathbf{(2)}$ If $\sigma=id_R$ and $\delta=0$ we get the concept of McCoy module, if only $\delta=0$, we get the concept of $\sigma$-skew McCoy module.
\par$\mathbf{(3)}$ A module $M_R$ is $(\sigma,\delta)$-skew McCoy if and only if for all $m(x)\in M[x;\sigma,\delta]$, $r_{R[x;\sigma,\delta]}(m(x))\neq 0\Rightarrow r_{R[x;\sigma,\delta]}(m(x))\cap R\neq 0.$
\end{remark}

An ideal $I$ of a ring $R$ is called $(\sigma,\delta)$-stable, if $\sigma(I)\subseteq I$ and $\delta(I)\subseteq I$.

\begin{proposition}\label{prop2}$\mathbf{(1)}$ Let $I$ be a nonzero right ideal of $R$. If $I$ is $(\sigma,\delta)$-stable then $R/I$ is an $R$-module $(\sigma,\delta)$-skew McCoy.
\par$\mathbf{(2)}$ For any index set $I$, if $M_i$ is an $(\sigma_i,\delta_i)$-skew McCoy as $R_i$-module for each $i\in I$, then $\prod_{i\in I}M_i$ is an $(\sigma,\delta)$-skew McCoy as $\prod_{i\in I}R_i$-module, where $(\sigma,\delta)=(\sigma_i,\delta_i)_{i\in I}$.
\par$\mathbf{(3)}$ Every submodule of an $(\sigma,\delta)$-skew McCoy module is $(\sigma,\delta)$-skew McCoy. In particular, if $I$ is a right ideal of an $(\sigma,\delta)$-skew McCoy ring then $I_R$ is $(\sigma,\delta)$-skew McCoy module.
\par$\mathbf{(4)}$ A module $M_R$ is $(\sigma,\delta)$-skew McCoy if and only if every finitely generated submodule of $M_R$ is $(\sigma,\delta)$-skew McCoy.
\end{proposition}

\begin{proof}$\mathbf{(1)}$ Let $m(x)=\sum_{i=0}^p\overline{m}_ix^i\in (R/I)[x;\sigma,\delta]$, where $\overline{m}_i=r_i+I\in R/I$ for all $i=0,1,\cdots,p$ and $r$ an arbitrary nonzero element of $I$. We have $m(x)r=\sum_{i=0}^p(r_i+I)\sum_{\ell=0}^if_{\ell}^i(r)x^{\ell}\in I[x;\sigma,\delta]$, because $f_{\ell}^i(r)\in I$ for all $\ell=0,1,\cdots,i$. Hence $m(x)r=\bar {0}$.
\par$\mathbf{(2)}$ Let $M=\prod_{i\in I}M_i$ and $R=\prod_{i\in I}R_i$ such that each $M_i$ is an $(\sigma_i,\delta_i)$-skew McCoy as $R_i$-module for all $i\in I$. Take $m(x)=(m_i(x))_{i\in I}\in M[x;\sigma,\delta]$ and $f(x)=(f_i(x))_{i\in I}\in R[x;\sigma,\delta]\setminus\{0\}$, where $m_i(x)=\sum_{s=0}^pm_i(s)x^s\in M_i[x;\sigma_i,\delta_i]$ and $f_i(x)=\sum_{t=0}^qa_i(t)x^t\in R_i[x;\sigma_i,\delta_i]$ for each $i\in I$. Suppose that $m(x)f(x)=0$, then $m_i(x)f_i(x)=0$ for each $i\in I$. Since $M_i$ is $(\sigma_i,\delta_i)$-skew McCoy, there exists $0\neq r_i\in R_i$ such that $m_i(x)r_i=0$ for each $i\in I$. Thus $m(x)r=0$ where $0\neq r=(r_i)_{i\in I}\in R$.
\par$\mathbf{(3)}$ and $(4)$ are obvious.
\end{proof}

\begin{proposition}\label{prop1}If $M_R$ is an $(\sigma,\delta)$-skew Armendariz module then it is $(\sigma,\delta)$-skew McCoy.
\end{proposition}

\begin{proof} Let $m(x)=\sum_{i=0}^pm_ix^i\in M[x;\sigma,\delta]$ and $g(x)=\sum_{j=0}^qb_jx^j\in R[x;\sigma,\delta]\setminus\{0\}$. Suppose that $m(x)g(x)=0$, then $m_ix^ib_jx^j=0$ for all $i,j$. Since $g(x)\neq 0$ then $b_{j_0}\neq 0$ for some $j_0\in\{0,1,\cdots,p\}$. Thus $m_ix^ib_{j_0}x^{j_0}=0$ for all $i$. On the other hand $m_ix^ib_{j_0}x^{j_0}=\sum_{\ell=0}^p(\sum_{i=\ell}^pm_if_{\ell}^i(b_{j_0}))x^{\ell+j_0}=0$, and so $\sum_{i=\ell}^pm_if_{\ell}^i(b_{j_0})=0$ for all $\ell=0,1,\cdots,p$. Thus $m(x)b_{j_0}=0$, therefore $M_R$ is $(\sigma,\delta)$-skew McCoy.
\end{proof}

By the next example, we see that the converse of Proposition \ref{prop1} does not hold.

\begin{example}\label{exp mcnotarm}Let $R$ be a reduced ring. Consider the ring
$$R_4=\set{\left(%
\begin{array}{ccccc}
  a & a_{12} & a_{13}& a_{14} \\
  0 & a & a_{23}& a_{24}  \\
  0 & 0 & a & a_{34}\\
  0 & 0 &0& a  \\
\end{array}%
\right)\mid a,a_{ij}\in R},$$
Since $R$ is reduced then it is right McCoy and so $R_4$ is right McCoy, by \cite[Proposition 2.1]{zhao/liu}. But $R_4$ is not Armendariz by \cite[Example 3]{kim/lee}.
\end{example}

A module $(\sigma,\delta)$-skew McCoy need not to be McCoy by \cite[Example 2.3(2)]{cui/2012}. Also, the following example shows that, there exists a module which is McCoy but not $(\sigma,\delta)$-skew McCoy.

\begin{example}\label{exp2}Let $\Z_2$ be the ring of integers modulo $2$, and consider the ring $R=\Z_2\oplus \Z_2$ with the usual addition and multiplication. Let $\sigma$ be an endomorphism of $R$ defined by $\sigma((a,b))=(b,a)$ and $\delta$ an $\sigma$-derivation of $R$ defined by $\delta((a,b))=(a,b)-\sigma((a,b))$. The ring $R$ is commutative reduced then it is McCoy. However, for $p(x)=(1,0)x$ and $q(x)=(1,1)+(1,0)x\in R[x;\sigma,\delta]$. We have $p(x)q(x)=0$, but $p(x)(a,b)\neq 0$ for any $0\neq (a,b)\in R$. Therefore, $R$ is not $(\sigma,\delta)$-skew McCoy. Also, $R$ is not $(\sigma,\delta)$-compatible, because $(0,1)(1,0)=(0,0)$, but $(0,1)\sigma((1,0))=(0,1)^2\neq (0,0)$ and $(0,1)\delta((1,0))=(0,1)(1,1)=(0,1)\neq (0,0)$.
\end{example}

\begin{lemma}\label{rem3} Let $M_R$ be an $(\sigma,\delta)$-compatible module. For any $m\in M_R$, $a\in R$ and nonnegative integers $i,j$. We have the following:
\par$\mathbf{(1)}$ $ma=0\Rightarrow m\sigma^i(a)=m\delta^j(a)=0$.
\par$\mathbf{(2)}$ $ma=0\Rightarrow m\sigma^i(\delta^j(a))=m\delta^i(\sigma^j(a))=0$.
\end{lemma}

\begin{proof}The verification is straightforward.
\end{proof}

If $M_R$ is an $(\sigma,\delta)$-compatible module then  $ma=0\Rightarrow mf_i^j(a)=0$ for any nonnegative integers $i,j$ such that $i\geq j$, where $m\in M_R$ and $a\in R$. For a subset $U$ of $M_R$ and $(\sigma,\delta)$ a quasi-derivation of $R$, the set of all skew polynomials with coefficients in $U$ is denoted by $U[x;\sigma,\delta]$.

\begin{lemma}\label{prop3}Let $M_R$ be a module and $(\sigma,\delta)$ a quasi-derivation of $R$. The following are equivalent:
\par$\mathbf{(1)}$ For any $U\subseteq M[x;\sigma,\delta]$, $(r_{R[x;\sigma,\delta]}(U)\cap R)[x;\sigma,\delta]=r_{R[x;\sigma,\delta]}(U)$.
\par$\mathbf{(2)}$ For any $m(x)=\sum_{i=0}^pm_ix^i\in M[x;\sigma,\delta]$ and $f(x)=\sum_{j=0}^qa_jx^j\in R[x;\sigma,\delta]$. If $m(x)f(x)=0$ implies $\sum_{\ell=i}^pm_{\ell}f_i^{\ell}(a_j)=0$ for all $i,j$.
\end{lemma}

\begin{proof}$(1)\Rightarrow (2)$. Let $m(x)=\sum_{i=0}^pm_ix^i\in M[x;\sigma,\delta]$ and $f(x)=\sum_{j=0}^qa_jx^j\in R[x;\sigma,\delta]$. If $m(x)f(x)=0$, we have $f(x)\in r_{R[x;\sigma,\delta]}(m(x))=(r_{R[x;\sigma,\delta]}(m(x))\cap R)[x;\sigma,\delta]$. Then $a_j\in r_{R[x;\sigma,\delta]}(m(x))$ for all $j$, so that $m(x)a_j=0$ for all $j$. But $m(x)a_j=0 \Leftrightarrow \sum_{\ell=i}^pm_{\ell}f_i^{\ell}(a_j)=0$ for all $0\leq i\leq p$. Thus $\sum_{\ell=i}^pm_{\ell}f_i^{\ell}(a_j)=0$ for all $i,j$.
\NL $(2)\Rightarrow (1)$. Let $U\subseteq M[x;\sigma,\delta]$, we have always $(r_{R[x;\sigma,\delta]}(U)\cap R)[x;\sigma,\delta]\subseteq r_{R[x;\sigma,\delta]}(U)$. Conversely, let $f(x)\in r_{R[x;\sigma,\delta]}(U)$ then by $(2)$, we have $Ua_j=0$ for all $j$ and so $a_j\in r_{R[x;\sigma,\delta]}(U)\cap R$. Therefore  $f(x)\in (r_{R[x;\sigma,\delta]}(U)\cap R)[x;\sigma,\delta]$.
\end{proof}

\begin{theorem}[McCoy's Theorem for module extensions]\label{theo mccoy}Let $M_R$ be a module and $N$ a nonzero submodule of $M[x;\sigma,\delta]$. If one of the equivalent conditions of Lemma \ref{prop3} is satisfied. Then $r_{R[x;\sigma,\delta]}(N)\neq 0$ implies $r_{R}(N)\neq 0$.
\end{theorem}

\begin{proof}Suppose that $r_{R[x;\sigma,\delta]}(N)\neq 0$, then there exists $0\neq f(x)=\sum_{i=0}^pa_ix^i\in r_{R[x;\sigma,\delta]}(N)$. But $r_{R[x;\sigma,\delta]}(N)=(r_{R[x;\sigma,\delta]}(N)\cap R)[x;\sigma,\delta]$ by Lemma \ref{prop3}. Therefore all $a_i$ are in $r_{R[x;\sigma,\delta]}(N)$, so $a_i\in r_R(N)$ for all $i$. Since $f(x)\neq 0$ then there exists $i_0\in \{0,1,\cdots p\}$ such that $0\neq a_{i_0}\in r_R(N)$. So that $r_{R}(N)\neq 0$.
\end{proof}

\begin{definition}\label{def}Let $M_R$ be a module and $\sigma$ an endomorphism of $R$. We say that $M_R$
satisfies the condition $(\mathcal{C_{\sigma}})$ if whenever $m\sigma(a)=0$ with $m\in M$ and $a\in R$, then $ma=0$.
\end{definition}

\begin{proposition}\label{prop/combine}Let $m(x)=\sum_{i=0}^{p}m_ix^i\in M[x;\sigma,\delta]$ and $f(x)=\sum_{j=0}^qa_jx^j$ $\in R[x;\sigma,\delta]$ such that $m(x)f(x)=0$. If one of the following conditions hold:
\par$\mathbf{(a)}$ $M_R$ is $(\sigma,\delta)$-skew Armendariz and satisfy the condition $(\mathcal{C_{\sigma}})$.
\par$\mathbf{(b)}$ $M_R$ is reduced and $(\sigma,\delta)$-compatible.
Then $m_ia_j=0$ for all $i,j$.
\end{proposition}

\begin{proof}\par$\mathbf{(a)}$ Since $M_R$ is $(\sigma,\delta)$-skew Armendariz then from $m(x)f(x)=0$, we get $m_ix^ia_jx^j=0$ for all $i,j$. But $m_ix^ia_jx^j=m_i\sum_{\ell=0}^if_{\ell}^i(a_j)x^{j+\ell}=m_i\sigma^i(a_j)x^{i+j}+Q(x)=0$
where $Q(x)$ is a polynomial in $M[x;\sigma,\delta]$ of degree strictly less than $i+j$. Thus $m_i\sigma^i(a_j)=0$, therefore $m_ia_j=0$ for all $i,j$.
\par$\mathbf{(b)}$ We will use freely the fact that, if $ma=0$ then $m\sigma^i(a)=m\delta^{j}(a)=mf_i^{j}(a)=0$ for any nonnegative integers $i,j$ with $j\geq i$. From $m(x)f(x)=0$, we have the following system of equations:
$$\;\qquad\qquad\qquad\qquad\qquad m_p\sigma^p(a_q)=0,\leqno(0)$$
$$m_p\sigma^p(a_{q-1})+m_{p-1}\sigma^{p-1}(a_q)+m_pf_{p-1}^p(a_q)=0,\qquad\qquad\leqno(1)$$
$$m_p\sigma^p(a_{q-2})+m_{p-1}\sigma^{p-1}(a_{q-1})+m_pf_{p-1}^p(a_{q-1})+m_{p-2}\sigma^{p-2}(a_q)+m_{p-1}f_{p-2}^{p-1}(a_q)\leqno(2)$$
$$\quad\qquad\qquad\qquad\quad\;\; +m_pf_{p-2}^p(a_q)=0,$$
$$m_p\sigma^p(a_{q-3})+m_{p-1}\sigma^{p-1}(a_{q-2})+
m_pf_{p-1}^p(a_{q-2})+m_{p-2}\sigma^{p-2}(a_{q-1})\leqno(3)$$
$$+m_{p-1}f_{p-2}^{p-1}(a_{q-1})+m_pf_{p-2}^p(a_{q-1})+m_{p-3}
\sigma^{p-3}(a_q)+m_{p-2}f_{p-3}^{p-2}(a_q)$$
$$\;\;\quad +m_{p-1}f_{p-3}^{p-1}(a_q)+m_pf_{p-3}^p(a_q)=0,$$
$$\qquad\qquad\qquad\qquad\vdots$$
$$\qquad\sum_{j+k=\ell}\;\;\sum_{i=0}^p\;
\sum_{k=0}^q(m_i\sum_{j=0}^if_j^i(a_k))=0,\leqno(\ell)$$
$$\qquad\qquad\qquad\qquad\vdots$$
$$\;\qquad\qquad\qquad\qquad\quad\sum_{i=0}^pm_i\delta^i(a_0)=0.\leqno(p+q)$$
From  equation $(0)$, we have $m_pa_q=0$ by $\sigma$-compatibility. Multiplying equation $(1)$ on the right hand by $a_q$, we get
$$m_p\sigma^p(a_{q-1})a_q+m_{p-1}\sigma^{p-1}(a_q)a_q+m_pf_{p-1}^p(a_q)a_q=0,\;\qquad\qquad\qquad\qquad\leqno(1')$$
Since $M_R$ is semicommutative, then $$m_pa_q=0\Rightarrow m_p\sigma^p(a_{q-1})a_q=m_pf_{p-1}^p(a_q)a_q=0.$$ By Lemma \ref{lemma banal}, equation $(1')$ gives $m_{p-1}a_q=0$. Also, by $(\sigma,\delta)$-compatibility, equation $(1)$ implies $m_p\sigma^p(a_{q-1})=0$, because $m_pa_q=m_{p-1}a_q=0$. Thus $m_pa_{q-1}=0$.
\par Summarizing at this point, we have $$m_pa_q=m_{p-1}a_q=m_pa_{q-1}=0\leqno(\alpha)$$
Now, multiplying equation $(2)$ on the right hand by $a_q$, we get
$$m_p\sigma^p(a_{q-2})a_q+m_{p-1}\sigma^{p-1}(a_{q-1})a_q+m_pf_{p-1}^p(a_{q-1})a_q+m_{p-2}\sigma^{p-2}(a_q)a_q\leqno(2')$$
$$+m_{p-1}f_{p-2}^{p-1}(a_q)a_q+m_pf_{p-2}^p(a_q)a_q=0,\;\;$$
With the same manner as above, equation $(2')$ gives $m_{p-2}\sigma^{p-2}(a_q)a_q=0$ and thus $m_{p-2}a_q=0\;(\beta)$. Also, multiplying equation $(2)$ on the right hand by $a_{q-1}$, we get $$m_p\sigma^p(a_{q-2})a_{q-1}+m_{p-1}\sigma^{p-1}(a_{q-1})a_{q-1}+m_pf_{p-1}^p(a_{q-1})a_{q-1}\leqno(2'')$$
$$+m_{p-2}\sigma^{p-2}(a_q)a_{q-1}+m_{p-1}f_{p-2}^{p-1}(a_q)a_{q-1}+m_pf_{p-2}^p(a_q)a_{q-1}=0$$
Equations $(\alpha)$ and $(\beta)$ implies $$\;0=m_p\sigma^p(a_{q-2})a_{q-1}=m_pf_{p-1}^p(a_{q-1})a_{q-1}=m_{p-2}\sigma^{p-2}(a_q)a_{q-1}$$
$$=m_{p-1}f_{p-2}^{p-1}(a_q)a_{q-1}=m_pf_{p-2}^p(a_q)a_{q-1}\qquad\qquad\qquad\qquad\;$$ Hence, equation $(2'')$ gives $m_{p-1}\sigma^{p-1}(a_{q-1})a_{q-1}=0$ and by Lemma \ref{lemma banal}, we get $m_{p-1}a_{q-1}=0\;(\gamma)$. Now, by equations $(\alpha)$,$(\beta)$ and $(\gamma)$, we get $m_{p-1}\sigma^{p-1}(a_{q-1})=m_pf_{p-1}^p(a_{q-1})=m_{p-2}\sigma^{p-2}(a_q)=m_{p-1}f_{p-2}^{p-1}(a_q)=m_pf_{p-2}^p(a_q)=0$. Therefore equation $(2)$ implies $m_p\sigma^p(a_{q-2})=0$, so that $m_pa_{q-2}=0$.
\par Summarizing at this point, we have $m_ia_j=0$ with $i+j\in \{p+q,p+q-1,p+q-2\}$. Continuing this procedure yields $m_ia_j=0$ for all $i,j$.
\end{proof}

\begin{lemma}\label{lemma banal}Let $M_R$ be an $(\sigma,\delta)$-compatible module, if $ma^2=0$ implies $ma=0$ for any $m\in M$ and $a\in R$. Then
\par$\mathbf{(1)}$ $m\sigma(a)a=0$ implies $ma=m\sigma(a)=0$.
\par$\mathbf{(2)}$ $ma\sigma(a)=0$ implies $ma=m\sigma(a)=0$.
\end{lemma}

\begin{proof}The proof is straightforward.
\end{proof}

According to Lee and Zhou \cite{lee/zhou}, a module $M_R$ is called $\sigma$-{\it reduced}, if for any $m\in M$ and $a\in R$. We have

\begin{enumerate}
  \item [$\mathbf{(1)}$] $ma=0$ implies $mR\cap Ma=0$.
  \item [$\mathbf{(2)}$] $ma=0$ if and only if $m\sigma(a)=0$.
\end{enumerate}

The module $M_R$ is called reduced if $M_R$ is $id_R$-reduced.

\begin{lemma}[{\cite[Lemma 1.2]{lee/zhou}}]\label{lemma zhou}The following are equivalent for a module $M_R$:
\begin{enumerate}
  \item [$\mathbf{(1)}$] $M_R$ is $\sigma$-reduced.
  \item [$\mathbf{(2)}$] The following three conditions hold: For any $m\in M$ and $a\in R$,
  \begin{enumerate}
    \item [$\mathbf{(a)}$] $ma=0$ implies $mRa=mR\sigma(a)=0$.
    \item [$\mathbf{(b)}$] $ma\sigma(a)=0$ implies $ma=0$.
    \item [$\mathbf{(c)}$] $ma^2=0$ implies $ma=0$.
  \end{enumerate}
\end{enumerate}
\end{lemma}

By Lemma \ref{lemma zhou}, a module $M_R$ is reduced if and only if it is semicommutative with $ma^2=0$ implies $ma=0$ for any $m\in M$ and $a\in R$.

\begin{corollary}[{\cite[Theorem 2.19]{moussavi/2012}}]Every $(\sigma,\delta)$-compatible and reduced module is $(\sigma,\delta)$-skew Armendariz.
\end{corollary}

\begin{proof}Clearly from Proposition \ref{prop/combine}(b).
\end{proof}

Let $M_R$ be a module and $(\sigma,\delta)$ a quasi derivation of $R$. We say that $M_R$ satisfies the condition $(*)$, if for any $m(x)\in M[x;\sigma,\delta]$ and $f(x)\in R[x;\sigma,\delta]$, $m(x)f(x)=0$ implies $m(x)Rf(x)=0$. A module $M_R$ which satisfies the condition $(*)$ is semicommutative. But the converse is not true, by the next example.

\begin{example}\label{ex2}Take the ring $R=\Z_2\oplus \Z_2$ with $(\sigma,\delta)$ as considered in Example \ref{exp2}. Since $R$ is commutative then the module $R_R$ is semicommutative. However, it does not satisfy the condition $(*)$. For $p(x)=(1,0)x$ and $q(x)=(1,1)+(1,0)x\in R[x;\sigma,\delta]$. We have $p(x)q(x)=0$, but $p(x)(1,0)q(x)=(1,0)+(1,0)x\neq 0$. Thus $p(x)Rq(x)\neq 0$.
\end{example}

\begin{theorem}\label{th2}If a module $M_R$ is $(\sigma,\delta)$-compatible and reduced, then it satisfies the condition $(*)$.
\end{theorem}

\begin{proof}Let $m(x)=\sum_{i=0}^pm_ix^i\in M[x;\sigma,\delta]$ and $f(x)=\sum_{j=0}^qa_jx^j\in R[x;\sigma,\delta]$, such that $m(x)f(x)=0$. By Proposition \ref{prop/combine}(b) and semicommutativity of $M_R$, we have $m_iRa_j=0$ for all $i$ and $j$. Moreover, compatibility implies $m_if_k^{\ell}(Ra_j)=0$ for all $i,j,k,\ell$. Therefore $m(x)Rf(x)=0$.
\end{proof}

Since the ring $R=\Z_2\oplus \Z_2$ is reduced, then from Example \ref{ex2}, we can see that the condition ``$(\sigma,\delta)$-compatible" in Theorem \ref{th2} is not superfluous.

\begin{proposition}\label{prop4}Let $M_R$ be an $(\sigma,\delta)$-compatible module which satisfies $(*)$. Suppose that for any $m(x)=\sum_{i=0}^pm_ix^i\in M[x;\sigma,\delta]$ and $f(x)=\sum_{j=0}^qa_jx^j\in R[x;\sigma,\delta]\setminus\{0\}$, $m(x)f(x)=0$. Then $m_ia_q^{p+1}=0$ for all $i=0,1,\cdots, p$.
\end{proposition}

\begin{proof}Let $m(x)=\sum_{i=0}^pm_ix^i\in M[x;\sigma,\delta]$ and $f(x)=\sum_{j=0}^qa_jx^j\in R[x;\sigma,\delta]\setminus\{0\}$, such that $m(x)f(x)=0$. We can suppose that $a_q\neq 0$. From $m(x)f(x)=0$, we get $m_p\sigma^p(a_q)=0$. Since $M_R$ is $(\sigma,\delta)$-compatible, we have $m_pa_q=0$ which implies $m_px^pa_q=0$. Since $m(x)f(x)=0$ implies $m(x)a_qf(x)=0$. Then $$0=(m_px^p+m_{p-1}x^{p-1}+\cdots+m_1x+m_0)(a_q^2x^q+a_qa_{q-1}x^{q-1}+\cdots+a_qa_1x+a_qa_0)$$
$$\;\;\;=(m_{p-1}x^{p-1}+\cdots+m_1x+m_0)(a_q^2x^q+a_qa_{q-1}x^{q-1}+\cdots+a_qa_1x+a_qa_0).\qquad\;$$
If we put $f'(x)=a_qf(x)$ and $m'(x)=\sum_{i=0}^{p-1}m_ix^i$ then we get $m_{p-1}a_q^2=0$. Continuing this procedure yields $m_ia_q^{p+1-i}=0$ for all $i=0,1,\cdots, p$. Consequently $m_ia_q^{p+1}=0$ for all $i=0,1,\cdots, p$.
\end{proof}

\begin{corollary}Let $M_R$ be an $(\sigma,\delta)$-compatible module over a reduced ring $R$. If $M_R$ satisfies $(*)$, then it is $(\sigma,\delta)$-skew McCoy.
\end{corollary}

\begin{proof}Let $m(x)=\sum_{i=0}^pm_ix^i\in M[x;\sigma,\delta]$ and $f(x)=\sum_{j=0}^qa_jx^j\in R[x;\sigma,\delta]\setminus\{0\}$, such that $m(x)f(x)=0$. We can suppose that $a_q\neq 0$. By Proposition \ref{prop4}, we have $m_ia_q^{p+1}=0$ for all $i=0,1,\cdots, p$. Since $M_R$ is $(\sigma,\delta)$-compatible, we get $m_ix^ia_q^{p+1}=m_i\sum_{\ell=0}^if_{\ell}^i(a_q^{p+1})x^{\ell}=0$ for all $i$. Hence $m(x)a_q^{p+1}=0$ where $a_q^{p+1}\neq 0$, because $R$ is reduced. Consequently $M_R$ is $(\sigma,\delta)$-skew McCoy.
\end{proof}

\begin{example}\label{ex2.2}Consider a ring of polynomials over $\Z_2$, $R=\Z_2[x]$. Let $\sigma\colon R\rightarrow R$ be an endomorphism
defined by $\sigma(f(x))=f(0)$. Then
\par$\mathbf{(1)}$ $R$ is not $\sigma$-compatible. Let $f=\overline{1}+x$, $g=x\in R$, we have $fg=(\overline{1}+x)x\neq 0$, however $f\sigma(g)=(\overline{1}+x)\sigma(x)=0$.
\par$\mathbf{(2)}$ $R$ is $\sigma$-skew Armendariz \cite[Example~5]{hong/2003}.
\end{example}

From Example \ref{ex2.2}, we see that the ring $R=\Z_2[x]$ is $\sigma$-skew McCoy because it is $\sigma$-skew Armendariz, but it is not $\sigma$-compatible. Thus the $(\sigma,\delta)$-compatibility condition is not essential to obtain $(\sigma,\delta)$-skew McCoyness.

\begin{example}[{\cite[Example 2.5]{louzari2}}]\label{ex5}Let $R$ be a ring, $\sigma$ an endomorphism of $R$ and $\delta$ be a $\sigma$-derivation of $R$. Suppose that $R$
is $\sigma$-rigid.
Consider the ring $$V_3(R)=\set{\left(%
\begin{array}{ccc}
  a & b&c\\
  0 & a& b \\
  0 & 0 & a\\
  \end{array}%
\right)\mid a,b,c\in R}.$$ The ring $V_3(R)$ is $(\overline{\sigma},\overline{\delta})$-skew McCoy, reduced and $(\overline{\sigma},\overline{\delta})$-compatible, and by Theorem \ref{th2}, it satisfies the condition $(*)$.
\end{example}

\section{$(\sigma,\delta)$-skew McCoyness of some matrix extensions}

For a nonnegative integer $n\geq 2$, let $R$ be a ring and $M$ a right $R$-module. Consider
$$S_n(R):=\set{\left(%
\begin{array}{ccccc}
  a & a_{12} & a_{13} & \ldots & a_{1n} \\
  0 & a & a_{23} & \ldots & a_{2n} \\
  0 & 0 & a & \ldots & a_{3n}\\
  \vdots & \vdots &\vdots&\ddots &\vdots \\
  0 & 0 & 0 & \ldots & a \\
\end{array}%
\right)\mid a,a_{ij}\in R}$$

and

$$S_n(M):=\set{\left(%
\begin{array}{ccccc}
  m & m_{12} & m_{13} & \ldots & m_{1n} \\
  0 & m & m_{23} & \ldots & m_{2n} \\
  0 & 0 & m & \ldots & m_{3n}\\
  \vdots & \vdots &\vdots&\ddots &\vdots \\
  0 & 0 & 0 & \ldots & m \\
\end{array}%
\right)\mid m,m_{ij}\in M}$$

Clearly, $S_n(M)$ is a right $S_n(R)$-module under the usual matrix addition operation and the following scalar product operation. For $U=(u_{ij})\in S_n(M)$ and $A=(a_{ij})\in S_n(R)$, $UA=(m_{ij})\in S_n(M)$ with $m_{ij}=\sum_{k=1}^nu_{ik}a_{kj}$ for all $i,j$. A quasi derivation $(\sigma,\delta)$ of $R$ can be extended to a quasi derivation $(\overline{\sigma},\overline{\delta})$ of $S_n(R)$ as follows: $\overline{\sigma}((a_{ij}))=(\sigma(a_{ij}))$ and $\overline{\delta}((a_{ij}))=(\delta(a_{ij}))$. We can easily verify that $\overline{\delta}$ is a $\overline{\sigma}$-derivation of $S_n(R)$.

\begin{theorem} A module $M_R$ is $(\sigma,\delta)$-skew McCoy if and only if $S_n(M)$ is $(\overline{\sigma},\overline{\delta})$-skew McCoy as an $S_n(R)$-module for any nonnegative integer $n\geq 2$.
\end{theorem}

\begin{proof} The proof is similar to \cite[Theorem 14]{baser/2009}.
\end{proof}

Now, for $n\geq 2$. Consider
$$V_n(R):=\set{\left(%
\begin{array}{cccccc}
  a_0 & a_1 & a_2 & a_3 & \ldots & a_{n-1} \\
  0 & a_0 & a_1 & a_2 & \ldots & a_{n-2} \\
  0 & 0 & a_0 & a_1 & \ldots & a_{n-3}\\
  \vdots & \vdots &\vdots& \vdots & \ddots &\vdots \\
  0 & 0 & 0 & 0 & \ldots & a_1 \\
  0 & 0 & 0 & 0& \ldots & a_0 \\
\end{array}%
\right)\mid a_0,a_1,a_2,\cdots,a_{n-1}\in R}$$

and

$$V_n(M):=\set{\left(%
\begin{array}{cccccc}
  m_0 & m_1 & m_2 & m_3 & \ldots & m_{n-1} \\
  0 & m_0 & m_1 & m_2 & \ldots & m_{n-2} \\
  0 & 0 & m_0 & m_1 & \ldots & m_{n-3}\\
  \vdots & \vdots &\vdots& \vdots & \ddots &\vdots \\
  0 & 0 & 0 & 0 & \ldots & m_1 \\
  0 & 0 & 0 & 0& \ldots & m_0 \\
\end{array}%
\right)\mid m_0,m_1,m_2,\cdots,m_{n-1}\in M}$$

With the same method as above, $V_n(M)$ is a right $V_n(R)$-module, and a quasi derivation $(\sigma,\delta)$ of $R$ can be extended to a quasi derivation $(\overline{\sigma},\overline{\delta})$ of $V_n(R)$. Note that $V_n(M)\cong M[x]/M[x](x^n)$ where $M[x](x^n)$ is a submodule of $M[x]$ generated by $x^n$ and $V_n(R)\cong R[x]/(x^n)$ where $(x^n)$ is an ideal of $R[x]$ generated by $x^n$.

\begin{proposition} A module $M_R$ is $(\sigma,\delta)$-skew McCoy if and only if $V_n(M)$ is $(\overline{\sigma},\overline{\delta})$-skew McCoy as an $V_n(R)$-module for any nonnegative integer $n\geq 2$.
\end{proposition}

\begin{proof}The proof is similar to that of \cite[Theorem 14]{baser/2009} or \cite[Proposition 2.27]{cui/2012}.
\end{proof}

\begin{corollary} For a nonnegative integer $n\geq 2$, we have:
\par$\mathbf{(1)}$ $M_R$ is $(\sigma,\delta)$-skew McCoy if and only if $M[x]/M[x](x^n)$ is $(\overline{\sigma},\overline{\delta})$-skew McCoy.
\par$\mathbf{(2)}$ $R$ is $(\sigma,\delta)$-skew McCoy if and only if $R[x]/(x^n)$ is $(\overline{\sigma},\overline{\delta})$-skew McCoy.
\par$\mathbf{(3)}$ $R$ is McCoy if and only if $R[x]/(x^n)$ is McCoy.
\end{corollary}

In the next, we define {\it skew triangular matrix modules} $V_n(M,\sigma)$, based on the definition of skew triangular matrix rings $V_n(R,\sigma)$ given by Isfahani \cite{isfahani/2011}. Let $\sigma$ be an endomorphism of a ring $R$ and $M_R$ a right $R$-module. For $n\geq 2$. Consider
$$V_n(R,\sigma):=\set{\left(%
\begin{array}{cccccc}
  a_0 & a_1 & a_2 & a_3 & \ldots & a_{n-1} \\
  0 & a_0 & a_1 & a_2 & \ldots & a_{n-2} \\
  0 & 0 & a_0 & a_1 & \ldots & a_{n-3}\\
  \vdots & \vdots &\vdots& \vdots & \ddots &\vdots \\
  0 & 0 & 0 & 0 & \ldots & a_1 \\
  0 & 0 & 0 & 0& \ldots & a_0 \\
\end{array}%
\right)\mid a_0,a_2,\cdots,a_{n-1}\in R}$$

and

$$V_n(M,\sigma):=\set{\left(%
\begin{array}{cccccc}
  m_0 & m_1 & m_2 & m_3 & \ldots & m_{n-1} \\
  0 & m_0 & m_1 & m_2 & \ldots & m_{n-2} \\
  0 & 0 & m_0 & m_1 & \ldots & m_{n-3}\\
  \vdots & \vdots &\vdots& \vdots & \ddots &\vdots \\
  0 & 0 & 0 & 0 & \ldots & m_1 \\
  0 & 0 & 0 & 0& \ldots & m_0 \\
\end{array}%
\right)\mid m_0,m_2,\cdots,m_{n-1}\in M}$$

Clearly $V_n(M,\sigma)$ is a right $V_n(R,\sigma)$-module under the usual matrix addition operation and the following scalar product operation.
$$\left(%
\begin{array}{cccccc}
  m_0 & m_1 & m_2 & m_3 & \ldots & m_{n-1} \\
  0 & m_0 & m_1 & m_2 & \ldots & m_{n-2} \\
  0 & 0 & m_0 & m_1 & \ldots & m_{n-3}\\
  \vdots & \vdots &\vdots& \vdots & \ddots &\vdots \\
  0 & 0 & 0 & 0 & \ldots & m_1 \\
  0 & 0 & 0 & 0& \ldots & m_0 \\
\end{array}%
\right)
\left(%
\begin{array}{cccccc}
  a_0 & a_1 & a_2 & a_3 & \ldots & a_{n-1} \\
  0 & a_0 & a_1 & a_2 & \ldots & a_{n-2} \\
  0 & 0 & a_0 & a_1 & \ldots & a_{n-3}\\
  \vdots & \vdots &\vdots& \vdots & \ddots &\vdots \\
  0 & 0 & 0 & 0 & \ldots & a_1 \\
  0 & 0 & 0 & 0& \ldots & a_0 \\
\end{array}%
\right)=$$
$$\left(%
\begin{array}{cccccc}
  c_0 & c_1 & c_2 & c_3 & \ldots & c_{n-1} \\
  0 & c_0 & c_1 & c_2 & \ldots & c_{n-2} \\
  0 & 0 & c_0 & c_1 & \ldots & c_{n-3}\\
  \vdots & \vdots &\vdots& \vdots & \ddots &\vdots \\
  0 & 0 & 0 & 0 & \ldots & c_1 \\
  0 & 0 & 0 & 0& \ldots & c_0 \\
\end{array}%
\right),\; \mathrm{where} $$
$c_i=m_0\sigma^{0}(a_i)+m_1\sigma^1(a_{i-1})+m_2\sigma^2(a_{i-2})+\cdots+m_i\sigma^{i}(a_0)$ for each $0\leq i\leq n-1$.
\par We denote elements of $V_n(R,\sigma)$ by $(a_0,a_1,\cdots,a_{n-1})$ and elements of $V_n(M,\sigma)$ by $(m_0,m_1,\cdots,m_{n-1})$. There is a ring isomorphism $\varphi\colon R[x;\sigma]/(x^n)\rightarrow V_n(R,\sigma)$ given by $\varphi(a_0+a_1x+a_2x^2+\cdots+a_{n-1}x^{n-1}+(x^n))=(a_0,a_1,a_2,\cdots,a_{n-1})$, and an abelian group isomorphism $\phi\colon M[x,\sigma]/M[x,\sigma](x^n)\rightarrow V_n(M,\sigma)$ given by $\phi(m_0+m_1x+m_2x^2+\cdots+m_{n-1}x^{n-1}+(x^n))=(m_0,m_1,m_2,\cdots,m_{n-1})$ such that $$\phi(N(x)A(x))=\phi(N(x))\varphi(A(x))$$ for any $N(x)=m_0+m_1x+m_2x^2+\cdots+m_{n-1}x^{n-1}+(x^n)\in M[x,\sigma]/M[x,\sigma](x^n)$ and $A(x)=a_0+a_1x+a_2x^2+\cdots+a_{n-1}x^{n-1}+(x^n)\in R[x;\sigma]/(x^n)$. The endomorphism $\sigma$ of $R$ can be extended to $V_n(R,\sigma)$ and $R[x;\sigma]$, and we will denote it in both cases by $\overline{\sigma}$.

\begin{theorem} A module $M_R$ is $\sigma$-skew McCoy if and only if $V_n(M,\sigma)$ is $\overline{\sigma}$-skew McCoy as an $V_n(R,\sigma)$-module for any nonnegative integer $n\geq 2$.
\end{theorem}

\begin{proof}We shall adapt the proof of \cite[Theorem 14]{baser/2009} to this situation. Note that $V_n(R,\sigma)[x,\overline{\sigma}]\cong V_n(R[x,\sigma], \overline{\sigma})$ and $V_n(M,\sigma)[x,\overline{\sigma}]\cong V_n(M[x,\sigma], \overline{\sigma})$. We only prove when $n=2$, because other cases can be proved with the same manner. Suppose that $M_R$ is $\sigma$-skew McCoy. Let $0\neq m(x)\in V_2(M,\sigma)[x,\overline{\sigma}]$ and $0\neq f(x)\in V_2(R,\sigma)[x,\overline{\sigma}]$ such that $m(x)f(x)=0$, where
$$m(x)=\sum_{i=0}^p
\left(\begin{array}{cc}
m_{11}^{(i)} & m_{12}^{(i)} \\
0 & m_{11}^{(i)} \\
\end{array}\right)x^i=
\left(\begin{array}{cc}
\sum_{i=0}^pm_{11}^{(i)}x^i & \sum_{i=0}^pm_{12}^{(i)}x^i \\
0 & \sum_{i=0}^pm_{11}^{(i)}x^i \\
\end{array}\right)=\left(\begin{array}{cc}
\alpha_{11} & \alpha_{12} \\
0 & \alpha_{11} \\
\end{array}\right)$$

$$f(x)=\sum_{j=0}^q
\left(\begin{array}{cc}
a_{11}^{(j)} & a_{12}^{(j)} \\
0 & a_{11}^{(j)} \\
\end{array}\right)x^j=
\left(\begin{array}{cc}
\sum_{j=0}^q
a_{11}^{(j)}x^j & \sum_{j=0}^q
a_{12}^{(j)}x^j \\
0 & \sum_{j=0}^q
a_{11}^{(j)}x^j \\
\end{array}\right)=
\left(\begin{array}{cc}
\beta_{11} & \beta_{12} \\
0 & \beta_{11} \\
\end{array}\right)$$
Then
$\left(\begin{array}{cc}
\alpha_{11} & \alpha_{12} \\
0 & \alpha_{11} \\
\end{array}\right)\left(\begin{array}{cc}
\beta_{11} & \beta_{12} \\
0 & \beta_{11} \\
\end{array}\right)=0$, which gives $\alpha_{11}\beta_{11}=0$ and $\alpha_{11}\beta_{12}+\alpha_{12}\overline{\sigma}(\beta_{11})=0$ in $M[x;\sigma]$. If $\alpha_{11}\neq 0$, then there exists $0\neq \beta\in \{\beta_{11},\beta_{12}\}$ such that $\alpha_{11}\beta=0$. Since $M_R$ is $\sigma$-skew McCoy then there exists $0\neq c\in R$ which satisfies $\alpha_{11}c=0$, thus $\left(\begin{array}{cc}
\alpha_{11} & \alpha_{12} \\
0 & \alpha_{11} \\
\end{array}\right)\left(\begin{array}{cc}
0 & c \\
0 & 0 \\
\end{array}\right)=\left(\begin{array}{cc}
0 & \alpha_{11}c \\
0 & 0 \\
\end{array}\right)=0$. If $\alpha_{11}=0$ then $\left(\begin{array}{cc}
0 & \alpha_{12} \\
0 & 0 \\
\end{array}\right)\left(\begin{array}{cc}
0 & c \\
0 & 0 \\
\end{array}\right)=0$, for any $0\neq c\in R$. Therefore, $V_2(M,\sigma)$ is $\overline{\sigma}$-skew McCoy.
\par Conversely, suppose that $V_2(M,\sigma)$ is an $\overline{\sigma}$-skew McCoy module. Let $0\neq m(x)=m_0+m_1x+\cdots+m_px^p\in M[x;\sigma]$ and $0\neq f(x)=a_0+a_1x+\cdots+a_qx^q\in R[x;\sigma]$, such that $m(x)f(x)=0$. Then $\left(\begin{array}{cc}
m(x) & 0 \\
0 & m(x) \\
\end{array}\right)\left(\begin{array}{cc}
f(x) & 0 \\
0 & f(x) \\
\end{array}\right)=\left(\begin{array}{cc}
m(x)f(x) & 0 \\
0 & m(x)f(x) \\
\end{array}\right)=0$, so there exists $0\neq \left(\begin{array}{cc}
a & b \\
0 & a \\
\end{array}\right)\in V_2(R,\sigma)$ such that $\left(\begin{array}{cc}
m(x) & 0 \\
0 & m(x) \\
\end{array}\right)\left(\begin{array}{cc}
a & b \\
0 & a \\
\end{array}\right)=0$, because $V_2(M,\sigma)$ is $\overline{\sigma}$-skew McCoy. Thus $m(x)a=m(x)b=0$, where $a\neq 0$ or $b\neq 0$. Therefore, $M_R$ is $\sigma$-skew McCoy.
\end{proof}

\begin{corollary} For a nonnegative integer $n\geq 2$, we have:
\par$\mathbf{(1)}$ $M_R$ is $\sigma$-skew McCoy if and only if $M[x;\sigma]/M[x;\sigma](x^n)$ is $\overline{\sigma}$-skew McCoy.
\par$\mathbf{(2)}$ $R$ is $\sigma$-skew McCoy if and only if $R[x;\sigma]/(x^n)$ is $\overline{\sigma}$-skew McCoy.
\par$\mathbf{(3)}$ $M_R$ is McCoy if and only if $M[x]/M[x](x^n)$ is McCoy.
\par$\mathbf{(4)}$ $R$ is McCoy if and only if $R[x]/(x^n)$ is McCoy.
\end{corollary}


\end{document}